\documentclass[11pt, a4paper]{article}
\usepackage{fullpage}

\usepackage{latexsym}
\usepackage{amsmath}
\usepackage{amssymb}
\usepackage{amsthm}
\usepackage{graphicx}
\usepackage{hyperref}
\usepackage{cite}

\newtheorem{theorem}{Theorem}[section]
\newtheorem{theorem*}{Theorem}
\newtheorem{corollary}[theorem]{Corollary}
\newtheorem{lemma}[theorem]{Lemma}

\newtheorem{claim}[theorem]{Claim}

\theoremstyle{definition}

\newtheorem{question}[theorem]{Question}

\newcommand{\abs}[1]{\left|#1\right|}                 

\newcommand{\F}{\mathbb{F}}

\newcommand{\R}{\mathbb{R}}
\newcommand{\C}{\mathbb{C}}

\renewcommand{\F}{\mathbb{F}}
\newcommand{\V}{\mathsf{V}}
\renewcommand{\R}{\mathsf{R}}
\renewcommand{\C}[2]{\binom{#1}{#2}}

\DeclareMathOperator{\sign}{sign}

\title {Square root Bound on the Least Power Non-residue using a Sylvester-Vandermonde Determinant}
\author{Michael Forbes \footnote{Massachusetts Institute of Technology} \and Neeraj Kayal \footnote{Microsoft Research India} \and Rajat Mittal \footnote{Rutgers University} \and Chandan Saha \footnote{Max Planck Institute for Informatics}}

\begin{document}


\date{}


\maketitle

\begin{abstract}
We give a new elementary proof of the fact that the value of the least $k^{th}$ power non-residue in an arithmetic progression $\{bn+c\}_{n=0,1...}$, over a prime field $\F_p$, is bounded by $7/\sqrt{5} \cdot b \cdot \sqrt{p/k} + 4b + c$. Our proof is inspired by the so called \emph{Stepanov method}, which involves bounding the size of the solution set of a system of equations by constructing a non-zero low degree auxiliary polynomial that vanishes with high multiplicity on the solution set. The proof uses basic algebra and number theory along with a determinant identity that generalizes both the Sylvester and the Vandermonde determinant. 
\end{abstract}

\section{Introduction}
Let $\F_p$ be the prime field with $p$ elements. An element $a \in \F_p$ is called a $k^{th}$ power non-residue if there is no $b \in \F_p$ such that $b^k = a$. Bounding the value of the least $k^{th}$ power non-residue in a prime field $\F_p$, where $k \mid p-1$, is a fundamental problem in number theory and algebra. It has an important application in finding roots over finite fields. For instance, it is known from the work of Vinogradov \cite{V72} (see also Proposition $7$ in \cite{E94}) that given a $k^{th}$ non-residue, all the $k^{th}$ power roots of an element $a \in \F_p$ i.e. all $x$ such that $x^{k} = a$, can be found in $(k \cdot \log p)^{O(1)}$ time. It is a major open problem in number theory to show that the least $k^{th}$ power non-residue is bounded by $(\log p)^{O(1)}$. Indeed, such a bound is already known under the powerful assumption of the Extended Riemann Hypothesis (ERH). It follows from the work of Ankeny \cite{A52} and Bach \cite{B82} that assuming ERH
 , the value of the least $k^{th}$ non-residue in $\F_p$ is bounded by $O(\log^2 p)$, where $k$ is a prime dividing $p-1$. However, such a strong bound is not yet shown without the assumption of any unproven conjecture. We now briefly mention the known results on ERH-free bounds on least power non-residues. 

\subsection{Earlier work}
The P\'{o}lya-Vinogradov inequality (see Chapter $23$ in \cite{D00}) states that
\begin{equation*}
\abs{\sum_{m+1 \leq x \leq m+n}{\chi(x)}} \leq \sqrt{p} \cdot \log p, 
\end{equation*}
where $\chi$ is a \emph{non-principal} character modulo $p$. Taking $\chi$ to be the quadratic character, it immediately follows that the least quadratic non-residue in $\F_p$ is bounded by $\sqrt{p}\log p$. In $1919$, this bound was improved by Vinogradov (see \cite{V54, V85}), who showed that the least quadratic non-residue in $\F_p$ is less than $p^{\frac{1}{2\sqrt{e}}} \log^2 p$. In a subsequent work, Vinogradov \cite{V27} also showed that if $k \mid p-1$ and $k > m^m$, where $m$ is an integer greater than $8$, then the least $k^{th}$ power non-residue is less than $p^{1/m}$ for all sufficiently large values of $p$. Later, in $1957$, Burgess \cite{B57} improved upon Vinogradov's result and showed that the least quadratic non-residue is in fact bounded by $p^{\frac{1}{4 \sqrt{e}} + \epsilon}$ for any small enough $\epsilon > 0$. A simple account of Burgess' theorem can be found in the work of Stepanov \cite{S75} (see also \cite{K68}). We note that the proofs of Vinogradov 
 and Burgess' results involve sophisticated analytic arguments on character sums. On the other hand, using purely elementary methods Brauer \cite{B32} showed that the length of the largest sequence of consecutive $k^{th}$ power residues or non-residues is bounded by $\sqrt{2p} + 2$. Later, Hudson \cite{H74} gave an elementary argument to show that the value of the smallest
$k^{th}$ power non-residue in an arithmetic progression $\{ bn + c\}_{n=0,1...}$ is bounded by $2^{11/4} b^{5/2} p^{2/5} + 6 b^3 p^{1/5} + 2 b^2$, if $p$ is sufficiently large. Surely, these bounds are worse than the best known bounds of Burgess and Vinogradov. Nevertheless, it is perhaps interesting to know how much elementary methods can achieve in proving non-trivial bounds for power residues and non-residues. 

\subsection{Our results}
We give a simple proof of the following fact.
\begin{theorem} \label{thm:pnrap}
The value of the least $k^{th}$ power non-residue in an arithmetic progression $\{ b n + c\}_{n=0,1, \ldots}$ over $\mathbb{F}_p$ is bounded by $7/\sqrt{5} \cdot b \cdot \sqrt{\frac{p-1}{k}} + 4b + c$.
\end{theorem}

\noindent Notice that, for $k \geq p^{1/5}$, the bound given by Theorem \ref{thm:pnrap} is better than the bounds shown by Hudson \cite{H74} and Brauer \cite{B32}. Our proof is inspired by the \emph{polynomial method}, which was introduced by Stepanov to give elementary proofs of many of the significant special cases of Weil's theorem on rational points on curves. The reader is encouraged to refer to the book by Schmidt \cite{S04} for an account of the elementary methods used in studying equations over finite fields. (For a quick introduction to some of the main results in this area refer to Tao's blog entry \cite{T09}.) 

The main idea behind Stepanov's method is to construct a non-zero auxiliary polynomial that vanishes with high multiplicity on the solution set of a system of equations. Now, if the degree of the auxiliary polynomial is also `small' then this can be used to upper bound the size of the solution set. We use this theme of bounding a solution set size via a low-degree auxiliary polynomial to give a new elementary proof of the square root bound on the least power non-residue/residue in any arithmetic progression. But, it turns out that the only `not so easy' part of our proof is showing that the auxiliary polynomial thus constructed is non-zero. We resolve this difficulty by using an interesting determinant identity that generalizes the determinant of both the Sylvester and the Vandermonde matrix. Proving this determinant identity constitutes the main technical contribution of our work. We hope that this identity on a generalized Sylvester-Vandermonde matrix is of independent interest and may find applications elsewhere.

\section{The Polynomial Method}

In this section, we describe our approach to proving Theorem \ref{thm:pnrap}. At the heart of our argument is the following lemma.

\begin{lemma} \label{lem:commonsol}
A system of univariate polynomials $\mathcal{S} = \{(x + a_i)^t - \theta_i\}_{1 \leq i \leq r}$, where $\theta_i, a_i \in \mathbb{F}_p$ and $a_i$'s are distinct, has at most $2t/(r-1) + 3$ common roots, if $r \leq 2/\sqrt{5} \cdot \sqrt{t} + 1$ and $p > 2t$.
\end{lemma}

\noindent Before we prove this lemma, let us at first see how it implies Theorem \ref{thm:pnrap}. (To keep the presentation simple, we avoid the use of the floor/ceiling notations. 
The analysis can be made more precise, at the cost of making the constant $4$ in Theorem \ref{thm:pnrap} and the constant $3$ in Lemma \ref{lem:commonsol} slightly worse.)

In Lemma \ref{lem:commonsol}, take $a_i = b \cdot (i-1)$, $t = (p-1)/k$ and $\theta_i = 1$ for all $1 \leq i \leq r$. Set $r = 2/\sqrt{5} \cdot \sqrt{t} + 1$. If the sequence of elements $bj + c, bj+c+b, bj+ c+ 2b, \ldots, bj+c+b(r-1)$ are $k^{th}$ power residues then surely, $bj + c$ is a common root of the system $\mathcal{S}$. By Lemma \ref{lem:commonsol}, there are at most $2t/(r-1) + 3$ common roots of $\mathcal{S}$. In the worst case, all these common roots can possibly be consecutive elements of the arithmetic progression $\{ bn + c\}_{n = 0, 1, \ldots}$. Therefore, the first index $m$ for which $bm + c$ is a $k^{th}$ power non-residue, can be at most $2t/(r-1) + 3 + r = 7/\sqrt{5} \cdot \sqrt{t} + 4$. The same argument can be used to prove a slightly general form of Theorem $\ref{thm:pnrap}$, as stated in the following corollary. 

\begin{corollary} \label{cor:conpnrap}
The length of the largest sequence of consecutive $k^{th}$ power residues or non-residues in an arithmetic progression $\{bn+c\}_{n=0,1,\ldots}$ is bounded by $7/\sqrt{5} \cdot \sqrt{p-1/k} + 4$.
\end{corollary}

\noindent The rest of this section and the following section (Section \ref{sec:fullrank}) are devoted to the proof of Lemma \ref{lem:commonsol}.

\subsection{Proof of Lemma \ref{lem:commonsol}}
The strategy we employ to bound the number of common roots of $\mathcal{S}$, denoted by $\nu(\mathcal{S})$ henceforth, is inspired by what is known as the `Stepanov method' (also called the `polynomial method'). The idea is to show the existence of a non-zero polynomial $F(x)$ of \emph{small} degree $N$ (say) such that if $\alpha$ is a common root of $\mathcal{S}$ then $\alpha$ is also a root of $F(x)$ with multiplicity $M$ (say). If this happens then we immediately know that $\nu(\mathcal{S})$ can be at most $\frac{N}{M}$. By making $N$ as small as possible and $M$ as large as possible, we can arrive at an upper bound for $\nu(\mathcal{S})$.

Let us see how to put this idea at work. Choose $F(x)$ to be of the form,
\begin{equation} \label{eqn:polyform}
F(x) = \sum_{i = 1}^{r}{G_i(x) (x+a_i)^{t+M-1+s}},
\end{equation}
where $G_i$'s are polynomials of degree at most $d$ (say) and $M$ is the multiplicity parameter mentioned above. The parameters $d$ and $M$ will be fixed eventually in terms of $t$ and $r$. Define $s$ as,
\begin{eqnarray*}
 s &=& 0 \text{,\hspace{0.1in} if $(r-1) \mid (t+M-1)$}  \\
   &=& (r-1) - ((t+M-1) \mod (r-1)) \text{,\hspace{0.1in} otherwise}. 
\end{eqnarray*}
The role of the parameter $s$ is to make $t+M-1+s$ perfectly divisible by $r-1$, a technical requirement for the analysis in Section \ref{sec:fullrank} to go through. Let us take a short digression and clarify a bit more the purpose of the parameter $s$. 

One might wonder as to why we do not assume, for the sake of simplicity, that $s=0$ and $t+M-1$ is divisible by $r-1$. At some point in our argument we need to establish linear independence of a certain linear system. If the coefficient matrix associated to the linear system is a square matrix then all we need to show is that the corresponding determinant is non-zero. It turns out, it can be shown that such a determinant is non-zero by using an identity involving \emph{derivatives} of the determinant function. Whereas, for a non-square system it is a little more tedious.

Coming back to the main flow of the proof, let us see what is required from the polynomial $F(x)$. Denote the $\ell^{th}$ derivative of $F$ with respect to $x$ by $F^{(\ell)}$ and let $T = t + M - 1 + s$. Also, $G^{(j)}$ denotes the $j^{th}$ derivative of $G$. If $\alpha$ is a root of the system $\mathcal{S}$ then we require $F^{(\ell)}(\alpha) = 0$ for all $0 \leq \ell \leq M-1$ since we want $\alpha$ to be a root of $F$ with multiplicity $M$. This means,
\begin{equation} \label{eqn:deriv_constraints}
F^{(\ell)}(\alpha) = \sum_{i=1}^{r}{\sum_{j=0}^{\ell}{ c_j(T) \cdot \theta_i \cdot G_i^{(j)}(\alpha) \cdot (\alpha + a_i)^{M-1+s-(\ell-j)} }} = 0,
\end{equation}
where $c_j(T) = \prod_{k=0}^{\ell - j - 1}{(T - k)}$ is a constant and $(\alpha + a_i)^t$ is evaluated to $\theta_i$ since $\alpha$ is a root of $\mathcal{S}$. Suppose that the coefficients of the polynomials $G_i$'s, in Equation \ref{eqn:polyform}, are variables. Also, treat the expression given in Equation \ref{eqn:deriv_constraints} as a polynomial in $\alpha$ of degree $(d + M - 1 + s - \ell)$ with coefficients as linear forms in the variables (that are coefficients of the $G_i$'s). By equating these coefficients to zeroes, we can ensure that $F^{(\ell)}(\alpha)$ is zero. Therefore, for any particular $\ell$, Equation \ref{eqn:deriv_constraints} imposes $(d + M + s - \ell)$ homogeneous linear constraints, yielding a total of $M \cdot (d + s) + \frac{M(M+1)}{2}$ homogeneous equations in $(d + 1) \cdot r$ variables (as $\ell$ runs from $0$ to $M-1$). Thus, in order that we get a nontrivial solution for the coefficients of $G_i$'s, it is sufficient to satisfy the the following condition,
\begin{equation} \label{cond:cond1}
M \cdot (d + s) + \frac{M(M+1)}{2} < (d+1) \cdot r.
\end{equation} 
Further, we also need to ensure that this solution is such that $F(x) \neq 0$. The degree of the polynomial $F(x)$ is $N = d + t + M - 1 + s$. If the number of variables $(d+1) \cdot r$ is greater than $d + t +  M + s$ then surely there is a nontrivial setting of the coefficients of $G_i$'s that makes $F(x) = 0$. However, such a situation can be possibly averted if we also put the restriction that
\begin{equation} \label{cond:cond2}
(d + 1) \cdot r \leq d + t + M + s.
\end{equation} 
Indeed, we show (in Section \ref{sec:fullrank}) that Condition \ref{cond:cond2} is sufficient to guarantee $F(x) \neq 0$ if the coefficients of the $G_i$'s are not all zeroes. To summarize, Condition \ref{cond:cond1} ensures that we are able to find nontrivial $G_i$'s by solving the homogeneous linear equations arising from Equation \ref{eqn:deriv_constraints}, for $0 \leq \ell \leq M-1$. Whereas, Condition \ref{cond:cond2} guarantees that the polynomial $F(x)$, defined in Equation \ref{eqn:polyform}, is non-zero if not all the $G_i$'s are zeroes - the proof of this appears in Section \ref{sec:fullrank}. 

Putting together Condition \ref{cond:cond1} and \ref{cond:cond2}, and using $D = d + 1$, we get the following overall condition to satisfy.
\begin{equation*}
M \cdot (D + s) + \frac{M(M-1)}{2} < D \cdot r \leq D + t + M + s - 1.
\end{equation*}
Since our objective is to minimize the quantity $\frac{N}{M} = \frac{D + t + M + s - 2}{M}$, we would like to minimize $D + t + M + s - 1$, which being lower bounded by $D \cdot r$, the best we could possibly do is to choose $D$ such that,
\begin{eqnarray} \label{eqn:Dvalue}
D \cdot r &=& D + t + M + s - 1 \nonumber \\
\Rightarrow D &=& \frac{t + M + s - 1}{r-1} 
\end{eqnarray} 

\noindent This setting of $D$ satisfies Condition \ref{cond:cond2}. Now, let us see how to satisfy Condition \ref{cond:cond1}. Choose $M = r/2$ and put $D = (t + M + s - 1)/(r-1)$ as in Equation \ref{eqn:Dvalue}. Using the fact that $s \leq r-1$ and then simplifying further, Condition \ref{cond:cond1} reduces to the following quadratic inequality:
\begin{equation*}
 5r^2 - 17r - (4t - 14) < 0.
\end{equation*}
It is easy to check that this is satisfied if $r \leq 2/\sqrt{5} \cdot \sqrt{t} + 1$. We are almost done. Recall that the maximum size of $\nu(\mathcal{S})$, the set of common solutions of $\mathcal{S}$, is bounded by $N/M$, where $N = \deg(F)$. Hence,
\begin{equation*}\
 \abs{\nu(\mathcal{S})} \leq \frac{N}{M} < \frac{D + t + M + s - 1}{M} = \frac{D \cdot r}{M} \text{\hspace{0.1in}(using Equation \ref{eqn:Dvalue})}
\end{equation*}
Since $M = r/2$, $\abs{\nu(\mathcal{S})} \leq 2D$. Once again, using the value of $D$ it is easy to derive that 
\begin{equation*}
 \abs{\nu(\mathcal{S})} \leq \frac{2t}{r-1} + 3.
\end{equation*}
This proves Lemma \ref{lem:commonsol} except the lemma:
\begin{claim} \label{clm:nullF}
If $D = \frac{t + M + s - 1}{r - 1}$ then $F(x) = 0$ if and only if $G_i(x) = 0$, for all $1 \leq i \leq r$.
\end{claim} 
\paragraph{} \noindent The next section is devoted to the proof of this statement. The main ingredient of the proof is an identity involving a generalized Sylvester-Vandermonde determinant. The condition ``$p > 2t$'' (in Lemma \ref{lem:commonsol}) also appears in this proof.

\section{A Generalized Sylvester-Vandermonde Determinant} \label{sec:fullrank}
Recall, from Equation \ref{eqn:polyform}, that $F(x)$ is defined as $F(x) = \sum_{i=1}^{r}{G_i(x) \cdot (x + a_i)^T}$, where $T = t + M + s - 1$. Suppose $G_i = \sum_{j=0}^{d}{c_{ij} x^j}$. Then,
\begin{equation*}
F(x) = \sum_{i=1}^{r}{\sum_{j=0}^{d}{c_{ij} x^j (x+a_i)^T}}. 
\end{equation*}
Proving Claim \ref{clm:nullF} essentially means proving this: if $F(x) = 0$ then $c_{ij} = 0$, for all $i$ and $j$. Suppose, on the contrary, that this is false. Then, the polynomials $\{ x^j (x + a_i)^T\}_{i,j}$, for $1 \leq i \leq r$ and $0 \leq j \leq d$, are $\F$-linearly dependent. In other words, the following matrix,
\begin{equation*}
\mathsf{V} = \begin{pmatrix}
			\C{T}{T}a_1^T	&\cdots		&\C{T}{T-d}a_1^{T-d}	&\cdots	&\cdots	&\C{T}{0}a_1^0		& 	&\\
					&\ddots		&\ddots			&\ddots	&\ddots	&\ddots			&\ddots	&\\
					&		&\C{T}{T}a_1^T		&\cdots	&\cdots	&\C{T}{d}a_1^{d}	&\cdots	& \C{T}{0}a_1^0\\
			\hline
			\C{T}{T}a_2^T	&\cdots		&\C{T}{T-d}a_2^{T-d}	&\cdots	&\cdots	&\C{T}{0}a_2^0		& 	&\\
					&\ddots		&\ddots			&\ddots	&\ddots	&\ddots			&\ddots	&\\
					&		&\C{T}{T}a_2^T		&\cdots	&\cdots	&\C{T}{d}a_2^{d}	&\cdots	& \C{T}{0}a_2^0\\
			\hline
			\vdots		&\vdots		&\vdots		&\vdots	&\vdots	&\vdots			&\vdots&\vdots\\
			\hline
			\C{T}{T}a_r^T	&\cdots		&\C{T}{T-d}a_r^{T-d}	&\cdots	&\cdots	&\C{T}{0}a_r^0		& 	&\\
					&\ddots		&\ddots			&\ddots	&\ddots	&\ddots			&\ddots	&\\
					&		&\C{T}{T}a_r^T		&\cdots	&\cdots	&\C{T}{d}a_r^{d}	&\cdots	& \C{T}{0}a_r^0
		\end{pmatrix}
\end{equation*}
must be singular. But, we show, in Lemma \ref{lem:SVidentity}, that $\mathsf{V}$ cannot be singular if $a_i$'s are distinct and $p > 2t$. This leads us to the necessary contradiction and hence a proof of Claim \ref{clm:nullF}. 

\paragraph{Remark} - Notice that, $\mathsf{V}$ is a square matrix since the number of rows $D \cdot r$ equals the number of columns $D + T$, by the choice of $D = T/(r-1)$ in Claim \ref{clm:nullF}. We call $\mathsf{V}$ a generalized Sylvester-Vandermonde matrix because when $r=2$, it becomes the Sylvester matrix of the two polynomials $(x + a_1)^T$ and $(x + a_2)^T$, whereas when $D=1$, it is the Vandermonde matrix (scaled appropriately).

\paragraph{} We now prove the following identity. 

\begin{lemma} [\textsf{Sylvester-Vandermonde identity}] \label{lem:SVidentity}
The $\det(\mathsf{V}) = C \cdot \prod_{1 \leq i < j \leq r}{(a_i - a_j)^{D^2}}$, where $C = \prod_{\ell = 0}^{T+d}{\C{T+d}{\ell}}/\prod_{j = 0}^{d}{\C{T+d}{j}^r}$.
\end{lemma}

\noindent It is not hard to check that $C \neq 0$ in $\F_p$, if $p > 2t$ (just use the facts that $s \leq r-2$ and $r$ is an integer less than or equal to $2/\sqrt{5} \cdot \sqrt{t} + 1$). 

\subsection{Proof of Lemma \ref{lem:SVidentity}}
First, we show that $\det(\V)$, viewed as a polynomial in $a_1, \ldots, a_r$, is divisible by $(a_1 - a_2)^{D^2}$. Then, by symmetry, $\det(\V)$ is also divisible by $(a_i - a_j)^{D^2}$ for every pair $(i,j)$ with $i < j$. Hence, $\det(\V)$ is divisible by $\prod_{1 \leq i < j \leq r}{(a_i - a_j)}^{D^2}$. By looking at the matrix $\V$, it is easy to infer that the highest degree of $a_1$ in $\det(\V)$ (once again, viewed as a polynomial in $a_1, \ldots, a_r$) is at most $D \cdot T$. Since the degree of $a_1$ in the expression $\prod_{1 \leq i < j \leq r}{(a_i - a_j)}^{D^2}$ is $D^2 \cdot (r-1) = D \cdot T$ (as $D = T/(r-1)$), $\det(\V)$ must be of the form $C \cdot \prod_{1 \leq i < j \leq r}{(a_i - a_j)^{D^2}}$, where $C$ is just a function of $T, d$ and $r$, but not the $a_i$'s. 

\paragraph{} In the proof, it will be more convenient if we express matrix $\V$ in terms of polynomials. Notice that the rows of $\V$ can be identified with the coefficient vectors of the polynomials $(x + a_1)^T, x(x + a_1)^T, \ldots, x^d(x+a_1)^T, (x+a_2)^T, x(x+a_2)^T, \ldots, x^d(x+a_2)^T, \ldots$ and so on. Let us abuse notations slightly and write $\V$ as,
\begin{equation} \label{eqn:redmatrix}
\V =
\begin{pmatrix}
(x+a_1)^T \\
x(x+a_1)^T \\
\vdots \\
x^{d} (x+a_1)^T \\
\hline 
(x+a_2)^T \\
\vdots \\
x^{d} (x+a_2)^T \\
\hline
\vdots \\
\hline
(x+a_r)^T \\
\vdots \\
x^{d} (x+a_r)^T \\
\end{pmatrix}
\hspace{0.2in}
\stackrel{\text{row operations}}{\longmapsto}
\hspace{0.2in}
\V' =
\begin{pmatrix}
(x+a_1)^T \\
(x+a_1)^{T+1} \\
\vdots \\
(x+a_1)^{T+d} \\
\hline 
(x+a_2)^T \\
\vdots \\
(x+a_2)^{T+d} \\
\hline
\vdots \\
\hline
(x+a_r)^T \\
\vdots \\
(x+a_r)^{T+d} \\
\end{pmatrix},
\end{equation}
meaning that $\V$ is formed by the coefficient vectors of these polynomials. Let $\mathsf{R}_{i j}$ be the row of $\V$ standing for the coefficient vector of $x^j(x+a_i)^T$. Consider the following row operations on $\V$. 
\begin{equation*}
\mathsf{R}_{i j} \mapsto \sum_{k=0}^{j}{\C{j}{k} \cdot a_i^k \cdot \mathsf{R}_{i \hspace{0.01in} j-k}}, \hspace{0.2in} \text{for $1 \leq i \leq r$ and $0 \leq j \leq d$}.
\end{equation*}
Equivalently, after the row operations, the coefficient vector of $x^j (x+a_i)^T$ gets replaced by that of the polynomial,
\begin{equation*}
 \sum_{k=0}^{j}{\C{j}{k} \cdot a_i^k \cdot x^{j-k} (x + a_i)^T} = (x + a_i)^{T+j}.
\end{equation*}
This leaves us with a transformed matrix $\V'$, as shown in Equation \ref{eqn:redmatrix}, such that $\det(\V) = \det(\V')$. To show that $(a_1 - a_2)^{D^2}$ divides $\det(\V')$, view $\det(\V') = f(a_1)$ as a polynomial in $a_1$. Let $f^{(\ell)}(a_1)$ denote the $\ell^{th}$ order derivative of $f(a_1)$ with respect to $a_1$. It is sufficient if we are able to show that $a_1 = a_2$ is a root of $f^{(\ell)}(a_1)$, for all $0 \leq \ell < D^2$.
\begin{claim} \label{clm:detstruct}
Let $f(a_1) = \det(\V')$ and $f^{(\ell)}(a_1) = \frac{\partial^{(\ell)}}{\partial a_1^{(\ell)}} f(a_1)$. Then $a_1 = a_2$ is a root of $f^{(\ell)}(a_1)$, for all $0 \leq \ell < D^2$.
\end{claim}
\begin{proof}
To prove this claim, we need the following identity involving derivatives of a determinant. Let $ A= (a_{i,j})$ be an $n \times n$ matrix whose entries are real functions of $x$. Then,
\begin{equation*}
\frac{d^\ell}{dx^\ell}\det(A)
= \sum_{\ell_1+ \ell_2 + \cdots + \ell_n = \ell} {\ell \choose \ell_1,\ell_2,...,\ell_n} \det \begin{pmatrix} \frac{ d^{\ell_1}}{dx^{\ell_1}}a_{1,1} & \frac{ d^{\ell_1}}{dx^{\ell_1}}a_{1,2} & \cdots
&\frac{d^{\ell_1}}{dx^{\ell_1}}a_{1,n} \cr \vdots & \vdots & & \vdots \cr \frac{d^{\ell_n}}{dx^{\ell_n}}a_{n,1} & \frac{d^{\ell_n}}{dx^{\ell_n}}a_{n,2} & \cdots &
\frac{d^{\ell_n}}{dx^{\ell_n}}a_{n,n}
\end{pmatrix},
\end{equation*}
where $ {\ell \choose \ell_1,\ell_2,...,\ell_r}$ is the multinomial coefficient. Now imagine applying this identity to $f^{(\ell)}(a_1)$, the $\ell^{th}$ order derivative of $\det(\V')$, where the entries of $\V'$ are viewed as functions of $a_1$. It is clear from Equation \ref{eqn:redmatrix} that except for the first $D$ rows of $\V'$, the rest are independent of the variable $a_1$. Denote by $\R_{i j}$, the row of $\V'$ generated by the coefficients of $(x+a_i)^{T + j}$. We write $\frac{d^{\ell}}{d a_1^{\ell}} \R_{i j}$ to mean the row formed by applying the operator $\frac{d^{\ell}}{d a_1^{\ell}}$ to every entry of $\R_{i j}$. Therefore, we have the following identity. (For economy of space, we switch to the transpose notation.)
\begin{equation*}
f^{(\ell)}(a_1) = \sum_{\ell_1+ \ell_2 + \cdots + \ell_D = \ell} {\ell \choose \ell_1,\ell_2,...,\ell_D} \det \left([\frac{d^{\ell_1}}{d a_1^{\ell_1}} \R_{1 0}, \ldots, \frac{d^{\ell_D}}{d a_1^{\ell_D}} \R_{1 d}, \R_{2 0}, \ldots, \R_{2 d}, \ldots, \R_{r 0}, \ldots, \R_{r d}]^T \right).
\end{equation*}
Notice one nice property of the polynomial representation of $\V'$: In the above equation, $\frac{d^{\ell_{j+1}}}{d a_1^{\ell_{j+1}}} \R_{1 j}$ is exactly the row formed by the coefficients of $\frac{d^{\ell_{j+1}}}{d a_1^{\ell_{j+1}}} (x + a_1)^{T + j} = c_j \cdot (x + a_1)^{T + j - \ell_{j+1}}$, where $c_j$ is a constant (depending only on $T$, $j$ and $\ell_{j+1}$). Now let us see how large $\ell$ needs to be so that $(a_1 - a_2)$ does not divide $f^{(\ell)}(a_1)$. 

Suppose $(a_1 - a_2) \nmid f^{(\ell)}(a_1)$. Then there exist a term $L$ in the above summation that is not divisible by $(a_1 - a_2)$. Let that term be identified by some tuple $(\ell_1, \ell_2, \ldots, \ell_D)$. Observe that this term,
\begin{eqnarray*}
L &=&  \det \left([\frac{d^{\ell_1}}{d a_1^{\ell_1}} \R_{1 0}, \ldots, \frac{d^{\ell_D}}{d a_1^{\ell_D}} \R_{1 d}, \R_{2 0}, \ldots, \R_{2 d}, \ldots, \R_{r 0}, \ldots, \R_{r d}]^T \right) \\
&=& \det \left( [c_0(x + a_1)^{T - \ell_1}, \ldots, c_d(x+a_1)^{T+d-\ell_D}, (x+a_2)^T, \ldots, (x+a_2)^{T+d}, \ldots]^T \right).
\end{eqnarray*}
If any of the exponents $\{T-\ell_1, T+1-\ell_2, \ldots, T+d-\ell_D \}$ is greater or equal to $T$ then $(a_1 - a_2) \mid L$; since otherwise some row $c_j (x+a_1)^{T+j-\ell_{j+1}}$ becomes equal to some other row $(x+a_2)^{T+k}$ (up to a multiple of $c_j$), when $a_1$ is replaced by $a_2$. Also, if any two of the exponents $\{T-\ell_1, T+1-\ell_2, \ldots, T+d-\ell_D \}$ are the same then $L = 0$. This leaves us with only one option - the set $\{T-\ell_1, T+1-\ell_2, \ldots, T+d-\ell_D \}$ is `dominated' by the set $\{T-1, T-2, \ldots, T-D\}$. (We say a set $S_1$ is \emph{dominated} by another set $S_2$ if for every element $e_1 \in S_1$ there is a unique element $e_2 \in S_2$ such that $e_1 \leq e_2$.) Therefore,
\begin{eqnarray*}
\sum_{j=0}^{d}{T + j - \ell_{j+1}} &\leq& \sum_{k=1}^{D}{T - k} \\
\Rightarrow \ell &\geq& DT + \frac{d(d+1)}{2} - DT + \frac{D(D+1)}{2} \\
&=& D^2 \text{\hspace{0.2in} (Taking $d = D-1$)} 
\end{eqnarray*}
\end{proof}
\noindent It follows from Claim \ref{clm:detstruct} and the discussion before that $\det(\V)$ is of the form $C \cdot \prod_{1 \leq i < j \leq r}{(a_i - a_j)^{D^2}}$, where $C$ is a constant depending only on $T, d$ and $r$. What remains to be done, in order to complete the proof of Lemma \ref{lem:SVidentity}, is to show that $C = \prod_{\ell = 0}^{T+d}{\C{T+d}{\ell}}/\prod_{j = 0}^{d}{\C{T+d}{j}^r}$. The proof of this is included in Appendix \ref{appdx}.

\section{Discussion}

\noindent Although, the square root bound of Corollary \ref{cor:conpnrap} is not the best known bound for this problem, it may be worthwhile exploring the `polynomial method' further to see if the bound can be strengthened, or if nontrivial bounds of some other related problems can be derived through it. Towards this, we have the following three questions in mind.  

\begin{question}
\emph{\textsf{(Strengthening Lemma \ref{lem:commonsol})}} Is it possible to give a better bound on the number of common solutions of the system $\mathcal{S}$ (defined in Lemma \ref{lem:commonsol}), perhaps by considering an auxiliary polynomial of the form,
\begin{equation*}
 F(x) = \sum_{\sigma \in S_\ell}{G_{\sigma}(x) \cdot \prod_{i \in \sigma}{(x+a_i)^{T+M-1}}},
\end{equation*}
where $S_\ell$ is the set of all $\ell$-tuples with $\ell$ distinct elements chosen from $\{1, \ldots r\}$? ($\abs{S_\ell} = {r \choose \ell}$).
\end{question}

\noindent Note, in our case, the auxiliary polynomial $F(x)$ is defined with $\ell = 1$.

\begin{question}
\emph{\textsf{(Simultaneous quadratic character)}} Is it possible to use our approach to show that for any pair of distinct elements $a, b \in \mathbb{F}_p$, the value of the largest possible $m$ for which $\chi((a + i) \cdot (b + i)) = 1$, for all $0 \leq i \leq m$, is $O(\sqrt{p} \log p)$? Here $\chi(a)$ denotes the quadratic character of $a$.  
\end{question}

\begin{question}
\emph{\textsf{(Least primitive element in $\mathbb{F}_p$)}} Is it possible to show that the value of the least primitive element in $\mathbb{F}_p$ is $O(\sqrt{p} \cdot \mathsf{poly}(\log p))$ using our approach?  
\end{question}

\section*{Acknowledgement}
This research was done when the authors F, M and S were interning at Microsoft Research India. The authors are thankful to MSR India for providing an excellent environment for research.

\bibliography{references}

\appendix

\section{The constant in Lemma \ref{lem:SVidentity}} \label{appdx}
Since the monomial $\prod_{i=1}^r a_i^{D^2(r-i)}$ has coefficient $1$ in the product $\prod_{1 \leq i<j \leq r} (a_i-a_j)^{D^2}$ (viewed as a polynomial in the $a_i$'s), the constant $C$ in Lemma \ref{lem:SVidentity} is the same as the coefficient of the monomial $\prod_{i=1}^r a_i^{D^2(r-i)}$ in $\det(\V)$.
\begin{claim}
The coefficient of the monomial $\prod_{i=1}^r a_i^{D^2(r-i)}$ in $\det(\V)$ is
$\prod_{\ell=0}^{T+d}\binom{T+d}{\ell}/\prod_{j=0}^{d}\binom{T+d}{j}^r$.
\end{claim}
\begin{proof}

By definition, $\det(\V)=\sum_{\sigma\in S_m}\sign(\sigma) \cdot \prod_{\ell \in [m]} v_{\ell,\sigma(\ell)}$, where $m = Dr$, $S_m$ is the symmetric group of degree $m$, and $v_{i,j}$ is the $(i,j)^{th}$ entry of $\V$. Note that, every product $\prod_{\ell \in [m]} v_{\ell,\sigma(\ell)}$ is a monomial in the $a_i$'s with an attached coefficient. We need to find out, which all permutations $\sigma$ give rise to the monomial $\prod_{i=1}^r a_i^{D^2(r-i)}$. 

Let both $R_i$ and $C_i$ denote the set $\{D(i-1)+1,\ldots,D(i-1)+D\}$, so that
$[Dr]=R_1\cup \cdots \cup R_r=C_1 \cup \cdots \cup C_r$. We think
of the $R_i$'s as partitioning the rows and the $C_i$'s as partitioning the columns of $\V$.
For instance, the rows with indices in $R_i$ contain only those terms involving the variable $a_i$. 
A crucial observation here is the following. The monomial $\prod_{i=1}^r a_i^{D^2(r-i)}$ is generated 
by exactly those permutations $\sigma$ that induce bijections between $R_i \leftrightarrow C_i$, for all $1 \leq i \leq r$.
This gives us a strategy to find the coefficient of $\prod_{i=1}^r a_i^{D^2(r-i)}$.

Define the matrix $M_i$ as the $D \times D$ submatrix of $\V$ which is induced by the rows $R_i$ and the columns $C_i$. Since each term in $\det(M_i)$ has the same degree in $a_i$, which is $D^2(r-i)$, the coefficient of $\prod_i a_i^{D^2(r-i)}$ can be obtained from the product $\prod_{1 \leq i \leq r} \det(M_i)$. 
		
\noindent Notice that $\det(M_i)=a_i^{D^2(r-i)} \cdot \det(H_i)$, where $H_i$ is the following matrix
formed by the binomial coefficients of the terms in $M_i$.  
\begin{equation*}
H_i =
	\begin{pmatrix}
	\binom{T}{T-D(i-1)}	&\binom{T}{T-D(i-1)-1} 		&\cdots	&\binom{T}{T-D(i-1)-d}\\
	\binom{T}{T-D(i-1)+1}	&\ddots 			&\ddots	&\binom{T}{T-D(i-1)-(d-1)}\\
	\vdots 			&\ddots				&\ddots	&\vdots\\
	\binom{T}{T-D(i-1)+d}	&\binom{T}{T-D(i-1)+(d-1)}	&\cdots	&\binom{T}{T-D(i-1)}
	\end{pmatrix}.
\end{equation*}
Therefore, the coefficient of $\prod_{i=1}^r a_i^{D^2(r-i)}$ in $\det(\V)$ is exactly $\prod_{1 \leq i \leq r} \det(H_i)$. For $1 < i < r$, each of the binomial coefficients in the matrix $H_i$ is non-degenerate, whereas for $i\in\{1,r\}$ it is easy to see that $\det(H_i)=1$ as $H_i$ is a triangular matrix with units along the diagonal. Suppose $1 < i < r$. After an appropriate row transformation, $H_i$ gets transformed to,

\begin{equation*}
H_i' =
\begin{pmatrix}
\binom{T+d}{T-D(i-1)+d}	&\binom{T+d}{T-D(i-1)+(d-1)} 		&\cdots	&\binom{T+d}{T-D(i-1)}\\
\binom{T+d-1}{T-D(i-1)+d}	&\ddots 			&\ddots	&\binom{T+d-1}{T-D(i-1)}\\
\vdots 			&\ddots				&\ddots	&\vdots\\
\binom{T}{T-D(i-1)+d}	&\binom{T}{T-D(i-1)+(d-1)}	&\cdots	&\binom{T}{T-D(i-1)}
\end{pmatrix},
\end{equation*}
so that $\det(H_i) = \det(H_i')$. Now we apply the following lemma, which we prove shortly, to find an expression for $\det(H_i')$.
\begin{lemma}
\label{lem:hankeldet}
For $n, m, \ell \in\mathbb{N}$, satisfying $\ell + m \le n$,
\begin{equation*}
\begin{vmatrix}
\binom{n+m}{\ell+m} & \binom{n+m}{\ell+m-1} & \cdots & \binom{n+m}{\ell}\\
\binom{n+m-1}{\ell+m} & \binom{n+m-1}{\ell+m-1} & \cdots & \binom{n+m-1}{\ell}\\
\vdots & \vdots &\ddots &\vdots\\
\binom{n}{\ell+m} & \binom{n}{\ell+m-1} & \cdots & \binom{n}{\ell}
\end{vmatrix}
=\frac{\prod_{j=0}^m\binom{n+m}{\ell+j}}{\prod_{j=0}^m\binom{n+m}{j}}
\end{equation*}
\end{lemma}

\noindent Taking $n \rightarrow T$, $m \rightarrow d$ and $\ell \rightarrow T - D(i-1)$, in the above lemma, we get 
\begin{equation}
\label{eq:hankel}
\det(H_i)
= \prod_{j=0}^d \frac{\binom{T+d}{T-D(i-1)+j}}{\binom{T+d}{j}}
= \prod_{j=0}^d \frac{\binom{T+d}{D(i-1)+j}}{\binom{T+d}{j}}.
\end{equation}
Note that the condition $\ell + m \le n$, in Lemma \ref{lem:hankeldet}, is satisfied after the substitution since $1 < i < r$. Also, the above formula \ref{eq:hankel} evaluates to $1$ for $i\in\{1,r\}$. Hence, the formula holds for all $1 \leq i \leq r$. Therefore, the coefficient of $\prod_{i=1}^r a_i^{D^2(r-i)}$ is,
\begin{equation*}
\prod_{i=1}^{r}{\det(H_i)}
=\prod_{i=1}^r \prod_{j=0}^d\frac{\binom{T+d}{D(i-1)+j}}{\binom{T+d}{j}}
=\frac{\prod_{\ell=0}^{T+d}\binom{T+d}{\ell}}{\prod_{j=0}^d\binom{T+d}{j}^r},
\end{equation*}
as claimed.
\end{proof}

It remains to prove Lemma \ref{lem:hankeldet}. 
\begin{proof} [Proof of Lemma \ref{lem:hankeldet}]
Index the rows from the bottom - the $0^{th}$ row $\R_0$ is the bottommost. Consider the row operation $\R_m \rightarrow (-1)^m \cdot \R_m + \sum_{k=0}^{m-1} (-1)^k \binom{m}{k} \frac{\binom{n+m}{\ell}}{\binom{n+k}{\ell}} \cdot \R_k$.
Then the value of the topmost row in the $i^{th}$ column (from the right, starting from zero) is 
\begin{align*}
\sum_{k=0}^m    \binom{n+k}{\ell+i}   (-1)^k   \binom{m}{k}  \frac{\binom{n+m}{\ell}}{\binom{n+k}{\ell}} =
\frac{\binom{n+m}{\ell}}{\binom{\ell+i}{\ell}} \sum_{k=0}^m   (-1)^k   \binom{m}{k} \binom{n-\ell+k}{i}  
\end{align*}
The claim is - the quantity inside the summation is zero for $i\in \{0,1,\cdots,m-1\}$. This is because of the following identity involving binomial coefficients (see page-169 in the book \cite{GKP89}). 
\begin{equation*}
\sum_{k=0}^m   (-1)^k   \binom{m}{k} \binom{s+k}{i} = (-1)^m \cdot \binom{s}{i-m},
\end{equation*}
which is zero as $\binom{s}{i-m} = 0$ when $i < m$ (by definition). For $i=m$, the above equation evaluates to $(-1)^m$. This means, after the transformation the value of the leftmost entry of the top row is $(-1)^m \cdot {\binom{n+m}{\ell}}/{\binom{\ell+m}{\ell}}$, whereas all the remaining entries of the row are zeroes. Recall that, in the row transformation we have multiplied the first row $\R_m$ by $(-1)^m$, and so remultiplying by $(-1)^m$ the top-left entry becomes simply ${\binom{n+m}{\ell}}/{\binom{\ell+m}{\ell}}$. Using this argument inductively on the minors, the determinant evaluates to,
\begin{equation*}
\prod_{j=0}^{m} \frac{\binom{n+j}{\ell}}{\binom{\ell+j}{\ell}} = \prod_{j=0}^m \frac{\binom{n+m}{\ell+j}}{\binom{n+m}{j}}.
\end{equation*}
The last equality is simple to verify.
\end{proof}

\end{document}